\documentclass[11pt,a4paper]{amsart}
\usepackage{amsmath,amsthm,amssymb,latexsym}
\usepackage{graphicx}

\newcommand{\re}{\text{\rm Re\,}}
\newcommand{\im}{\text{\rm Im\,}}

\newcommand{\bd}{{\mathbb{D}}}
\newcommand{\bn}{{\mathbb{N}}}
\newcommand{\br}{{\mathbb{R}}}

\newcommand{\bc}{{\mathbb{C}}}

\newcommand{\bt}{{\mathbb{T}}}

\newcommand{\cb}{{\mathcal{B}}}
\newcommand{\cf}{{\mathcal{F}}}

\newcommand{\cm}{{\mathcal{M}}}

\newcommand{\css}{{\mathcal{S}}}

\renewcommand{\a}{\alpha}
\renewcommand{\b}{\beta}

\newcommand{\s}{\sigma}

\newcommand{\p}{\varphi}

\renewcommand{\d}{\delta}

\renewcommand{\o}{\omega}

\newcommand{\g}{\gamma}
\renewcommand{\gg}{\Gamma}

\newcommand{\z}{\zeta}

\newcommand{\thth}{\Theta}

\newcommand{\lp}{\left(}
\newcommand{\rp}{\right)}
\newcommand{\lt}{\left\{}
\newcommand{\rt}{\right\}}

\DeclareMathOperator{\lc}{span}

\DeclareMathOperator{\supp}{supp}

\allowdisplaybreaks
\numberwithin{equation}{section}

\newtheorem{theorem}{Theorem}[section]

\newtheorem{lemma}[theorem]{Lemma}
\newtheorem{corollary}[theorem]{Corollary}
\newtheorem{proposition}[theorem]{Proposition}

\theoremstyle{definition}
\newtheorem{definition}[theorem]{Definition}
\newtheorem{remark}[theorem]{Remark}

\newtheorem{example}[theorem]{Example}

\begin{document}

\title[Extreme points]
{Extreme points in isometric embedding problem for model spaces}
\author[L. Golinskii]{L. Golinskii}

\address{B. Verkin Institute for Low Temperature Physics and
Engineering, 47 Science ave., Kharkiv 61103, Ukraine}
\email{golinskii@ilt.kharkov.ua}

\date{\today}

\keywords{model spaces; Schur functions; Blaschke products; inner functions; Clark measures; Nehari problem; minifunctions}
\subjclass[2010]{30J05, 30J10, 46B20}

\maketitle

\begin{abstract}
In 1996 A. Aleksandrov solved the isometric embedding problem for the model spaces $K_\Theta$ with an arbitrary
inner function $\Theta$. We find all extreme points of this convex set of measures in the case when $\Theta$ is a finite
Blaschke product, and obtain some partial results for generic inner functions.
\end{abstract}

\section*{Introduction}
\label{s0}

In \cite{al96} A. Aleksandrov settled the isometric embedding problem  for the model spaces $K_\thth:=H^2\ominus \thth H^2$.
Precisely, let $\Theta$ be an arbitrary, nonconstant inner function on the unit disk $\bd$, i.e., $\thth$ belongs to the unit ball of $H^\infty$
(the Schur class $\css$), and $|\thth|=1$ a.e. on the unit circle $\bt$. Denote by $\cm_+(\bt)$ the class of all finite, positive,
Borel measures $\bt$. The problem is to describe a collection $\cm(\thth)\subset\cm_+(\bt)$ of measures so that the identity operator 
(embedding) of the model space $K_\thth$ to the space $L^2_\s(\bt)$ is isometric. In other words, the equality
\begin{equation*}
\langle f,g\rangle_{\s}:=\int_{\bt} f(t)\overline{g(t)}\,\s(dt)=\int_{\bt} f(t)\overline{g(t)}\,m(dt)=\langle f,g\rangle_{m}, \ \
f,g\in K_\thth\cap C(\bt)
\end{equation*}
holds for each continuous functions $f,g\in K_\Theta$. Here $m$ is the normalized Lebesgue measure on $\bt$. As it turns out \cite{al95}, for such measures and
for each $f\in K_\thth$ the boundary values exist $\s$ a.e. on $\bt$, so the equality can be extended to the whole model space $K_\thth$.

The result of Aleksandrov looks as follows.

{\bf Theorem A}. {\it $\s\in\cm(\thth)$ if and only if there is a unique pair $(\b, \o)$ with a real number $\b$ and a Schur function $\o\in\css$ so that}
\begin{equation}\label{thal}
\frac{1+\thth(z)\o(z)}{1-\thth(z)\o(z)}=i\b+\int_{\bt}\frac{t+z}{t-z}\,\s(dt).
\end{equation}

For an alternative proof see \cite[Section 11.7]{GMR16}.

Relation \eqref{thal} can be viewed as a counterpart of the Nevanlinna parametrization for each indeterminate Hamburger moment problem, 
see \cite[Theorem 3.2.2] {Akh65} and Section \ref{s3} below.

\begin{remark} \label {rem1}
The function $\o$ in \eqref{thal} is an independent parameter, which runs over the class $\css$. Both $\b$ and $\s$ in \eqref{thal} are
uniquely determined by~$\o$,
$$\b=\frac{2\,\im\bigl(\o(0)\thth(0)\bigr)}{|1-\thth(0)\o(0)|^2}. $$
Conversely, if two triplets $\{\o_j,\b_j,\s\}$, $j=1,2$, satisfy \eqref{thal}, then $\o_1=\o_2$ and $\b_1=\b_2$. Indeed, we have
\begin{equation*}
\begin{split}
\frac{1+\thth\o_2}{1-\thth\o_2}-\frac{1+\thth\o_1}{1-\thth\o_1} &=i(\b_2-\b_1), \\
2\thth(\o_2-\o_1) &=i(\b_2-\b_1)(1-\thth\o_2)(1-\thth\o_1).
\end{split}
\end{equation*}
The function on the right side is outer, whereas the function on the left side has a nontrivial inner factor. So, $\o_1=\o_2$ and $\b_1=\b_2$, as claimed.
For instance, $\s=m$ enters the only triplet $\{0,0,m\}$. 

So, the equality \eqref{thal} generates a bijection $\cf$
\begin{equation}\label{bijec}
\cf:\ \css\to\cm(\thth), \quad \cf(\o)=\s.
\end{equation}
\end{remark}
\noindent
It is clear from \eqref{thal}, that the restriction of the map $\cf$ on the set of inner functions produces exactly the set of all singular measures in $\cm(\thth)$.

The isometric embedding problem has a long history, and there are at least two predecessors of A. Aleksandrov. L. de Branges \cite[Theorem 32]{Bra68}
solved the problem for meromorphic inner functions on the upper half-plane. For some follow-up results, including discussion of extreme points for the set
of de Branges measures, see \cite{bar}. Later on D. Sarason \cite{sar} proved the result
on the isometric embedding for inner functions $\thth$ with $\thth(0)=0$ and for measures of the form $|f|^2m$, $f\in H^2$. Since such measures
form a dense set in $\cm(\thth)$, the result of Aleksandrov can be deduced from that of Sarason (at least for $\thth(0)=0$).
So it seems reasonable referring to the measures from $\cm(\thth)$ as the {\it Aleksandrov--Sarason measures}.

Relations \eqref{thal} with unimodular constants $\o=\a\in\bt$
\begin{equation}\label{thal1}
\frac{1+\a\thth(z)}{1-\a\thth(z)}=i\b_\a+\int_{\bt}\frac{t+z}{t-z}\,\s_\a(dt)
\end{equation}
are well known in the theory of the model spaces \cite[Chapter 9]{CiMaR}, \cite[Chapter~11]{GMR16}. The measures $\s_\a$ in \eqref{thal1}
are the {\it Clark measures} following D. Clark \cite{cla}. Given $n=0,1,\ldots$ a measure $\s\in\cm(\thth)$ will be called a {\it Clark measure
of order} $n$ if the corresponding parameter $\o$ in \eqref{thal} is a finite Blaschke product (FBP) of order $n$.

The case $\thth(z)=z$ in \eqref{thal} arises in Geronimus' approach to the
theory of orthogonal polynomials on the unit circle \cite[Chapter 3]{Si1}. The model space is now the one dimensional space of constant functions,
and $\cm(\thth)$ is the set of all probability measures on $\bt$. 

The set $\cm(\thth)$ is easily seen to be a convex, compact in *-weak topology of the space $\cm_+(\bt)$ set. The study of the set 
$\cm_{ext}(\thth)$ of extreme points for $\cm(\thth)$ seems quite natural. This is exactly
the problem we address here. A point $\s\in\cm(\thth)$ is said to be an {\it extreme point of} $\cm(\thth)$ if
\begin{equation}\label{ext}
\s=\frac{\s_1+\s_2}2\,, \quad \s_j\in\cm(\thth) \ \ \Rightarrow \ \ \s_1=\s_2=\s.
\end{equation}
To paraphrase, there is no nontrivial representation of $\s$ as a convex linear combination of two Aleksandrov--Sarason measures.

We say that a measure $\s\in\cm_+(\bt)$ has finite support if
$$ \s=\sum_{j=1}^p s_j\d(t_j), \quad s_j>0, \qquad  \supp\s=\{t_j\}_{j=1}^p, \quad t_j=t_j(\s), $$
and write $|\supp\s|=p$ for the cardinality of the support. Denote by $\cm_f(\thth)$ the set of all measures in $\cm(\thth)$ with finite support.
It is clear from \eqref{thal} that $\cm_f(\thth)$ is nonempty if and only if both $\thth$ and $\o$ are FBP's.

The model space $K_\thth$ is finite dimensional, $\dim K_\thth<\infty$, if and only if $\thth=B$ is a FBP.
Our main result concerns this situation.

\begin{theorem}\label{mainth}
Let $B$ be a FBP of order $n\ge1$. A measure $\s\in\cm_{ext}(B)$ if and only if $\s\in\cm_f(B)$ and
\begin{equation}\label{mainres}
n\le |\supp\s|\le 2n-1.
\end{equation}
\end{theorem}

When studying subclasses of $\cm(\thth)$ it is sometimes reasonable to go to the $\cf$-preimages and consider the corresponding subclasses of the
Schur class instead. That is what we are going to do when dealing with the class $\cm_{ext}(\thth)$.

Denote by $\css_{ext}(\thth)\subset\css$ the set of extreme Schur functions
\begin{equation}\label{param}
\css_{ext}(\thth)=\{\o\in\css: \ \ \cf(\o)\in\cm_{ext}(\thth)\}.
\end{equation}
The above result has an equivalent form.

\begin{theorem}\label{mainth1}
Let $B$ be a FBP of order $n\ge1$. The set $\css_{ext}(B)$ agrees with the set of all FBP's of the order at most $n-1$.
\end{theorem}

The case of generic inner functions $\thth$ is much more delicate. We supplement the above result with the following

\begin{theorem}\label{finblext}
Let $\thth$ be a nonconstant inner functions, which is not a FBP. Then each FBP belongs to $\css_{ext}(\thth)$. Equivalently, each
Clark measure of order $n=0,1,\ldots$ belongs to $\cm_{ext}(\thth)$.
\end{theorem}

We say that an inner function $\p$ is a divisor of $\thth$ if $\thth/\p$ is again the inner function.

\begin{theorem}\label{divis}
Let $\thth$ be an arbitrary, nonconstant inner function. Then each divisor of  $\thth$, distinct from $\thth$,
belongs to $\css_{ext}(\thth)$, but $\thth\notin\css_{ext}(\thth)$.
\end{theorem}

The case of  unimodular constant divisors corresponds to the Clark measures.

\begin{corollary}\label{alcl}
Let $\thth$ be an arbitrary, nonconstant inner function. Then the Clark measures $\s_\a\in\cm_{ext}(\thth)$ for all $\a\in\bt$.
\end{corollary}

We examine the class $\cm_f(B)$ of measures with finite support in Section~\ref{s1}, and prove Theorem \ref{mainth} in Section \ref{s2}. In
Section \ref{s3}, given an inner function $\thth$, we introduce a binary operation $\css\times\css\to\css$ ($\thth$-product). As it turns out,
a Schur function $\o\notin\css_{ext}(\thth)$ if and only if $\o$ admits a nontrivial factorization with respect to the $\thth$-product.
Thereby, the ``$\thth$-prime" functions $\o$ constitute the class $\css_{ext}(\thth)$.
The results of Theorems \ref{finblext} and \ref{divis} are obtained along this line of reasoning, by using some uniqueness conditions for
the classical Nehari problem.

So far we have met only inner functions $\o$ in $\css_{ext}(\thth)$. By Theorem \ref{mainth1}, for FBP's $B$
all functions in $\css_{ext}(B)$ are inner. We show that in the opposite case, that is, $\dim K_\thth=\infty$, there are non-inner function
$\o\in\css_{ext}(\thth)$. In other words, the set $\cm_{ext}(\thth)$ contains measures with a nontrivial absolutely continuous part.

\section{Some properties of the class $\cm_f(B)$}
\label{s1}

Given a FBP $B$ of order $n$, we denote by
$$ \{(z_1,r_1), (z_2,r_2), \ldots, (z_d,r_d)\}, \qquad z_i\not=z_j, \quad i\not=j, \quad r_j\in\bn, $$
the set of its zeros, so that
$$ B(z):=\prod_{k=1}^d \lp\frac{|z_k|}{z_k}\,\frac{z_k-z}{1-\bar z_k z}\rp^{r_k}\,,
\qquad \deg B=r_1+\ldots+r_d=n. $$
The model space
\begin{equation}\label{finmodsp}
K_B:=H^2\ominus BH^2=\lt h(z)=\frac{P(z)}{\prod_{j=1}^d (1-\bar z_j z)^{r_j}}\,, \quad \deg P\le n-1 \rt,
\end{equation}
is the finite dimensional space of all rational functions with the poles at the points $1/\bar z_j$ of degree at most $r_j$, $\dim K_B=n$.
The case $z_d=0$, i.e., $B(0)=0$, will be of particular concern. Now
$$ K_B=\lt h(z)=\frac{P(z)}{\prod_{j=1}^{d-1} (1-\bar z_j z)^{r_j}}\,, \quad \deg P\le n-1 \rt, $$
and the monomials $1,z,\ldots,z^{r_d-1}\in K_B$. Put
$$ \p_0(z)=1, \quad \p_k(z):=\frac1{1-\bar z_k z}\,, \quad k=1,2,\ldots,d-1, \quad \p_d(z)=z, $$
so the standard basis in $K_B$ is
\begin{equation}\label{basis1}
\{\p_1,\p_1^2, \ldots, \p_1^{r_1}; \ldots; \p_{d-1},\p_{d-1}^2, \ldots, \p_{d-1}^{r_{d-1}}; \p_d,\ldots,\p_d^{r_d-1};\p_0\}.
\end{equation}
It seems reasonable rearranging these functions in a unique sequence $\{e_l\}_{l=1}^n$, $e_n=1$.

The following result is a consequence of Theorem A, but we give a simple, direct proof.

\begin{proposition}\label{pr1}
The support of each measure $\s\in\cm(B)$ contains at least $n$ points.
\end{proposition}
\begin{proof}
If $|\supp\s|\le n-1$, then $\dim L^2_\s(\bt)\le n-1$, and the functions $\{e_l\}_{l=1}^n$ in \eqref{basis1} are linearly dependent in $L^2_\s(\bt)$, so
$$ \det \|\langle e_j,e_k\rangle_\s\|_{j,k=1}^n=0. $$
On the other hand, the same system is linearly independent in $L^2_m(\bt)$, so $\det \|\langle e_j,e_k\rangle_m\|_{j,k=1}^n\not=0$.
The contradiction completes the proof.
\end{proof}

As we mentioned in Introduction, a measure $\s\in\cm_f(B)$ if and only if $\o=I^{(-1)}(\s)$ is a FBP. Moreover, $|\supp\s|=n+\deg\o$,
so $|\supp\s|=n$ if and only if $\s=\s_\a$ is the Clark measure \eqref{thal1}.

It is not hard to display $\s\in\cm_f(B)$ explicitly in terms of the corresponding parameters $\o$ and $B$. Indeed, \eqref{thal} now takes the form
\begin{equation}\label{alfin}
\frac{1+B(z)\o(z)}{1-B(z)\o(z)}=i\b+\sum_{k=1}^p \frac{t_k+z}{t_k-z}\,s_k,
\end{equation}
and
\begin{equation}\label{support}
\supp\s=\{t_j\}_{j=1}^p: \quad B(t_j)\o(t_j)=1, \quad j=1,2,\ldots,p.
\end{equation}
The weights $s_j$ can be determined from the limit relations
$$ 2t_qs_q=\bigl(1+B(t_q)\o(t_q)\bigr)\,\lim_{z\to t_q}\frac{t_q-z}{1-B(z)\o(z)}=\frac2{[B\o]'(t_q)}\,,  $$
or, in view of \eqref{support},
$$
\frac1{s_q}=t_q[B\o]'(t_q)=t_q\,\frac{B'(t_q)}{B(t_q)}+t_q\,\frac{\o'(t_q)}{\o(t_q)}\,. $$
A computation of the logarithmic derivative of a FBP is standard
$$ \frac{B'(z)}{B(z)}=\sum_{k=1}^d r_k\,\frac{1-|z_k|^2}{(1-\bar z_k z)(z-z_k)}\,. $$
So,
\begin{equation}\label{weights}
\frac1{s_q}=\sum_{k=1}^d r_k\,\frac{1-|z_k|^2}{|t_q-z_k|^2}+\sum_{j=1}^m \frac{1-|w_j|^2}{|t_q-w_j|^2}\,,
\end{equation}
where $w_1,\ldots,w_m$ are all zeros (counting multiplicity) of $\o$ in \eqref{alfin}.

The relation \eqref{weights} provides an answer to the following ``extremal mass problem'': given a point $\tau\in\bt$, find a measure
$\s_{\max}\in\cm_f(B)$ so that
$$ \s_{\max}\{\tau\}=\max\{\s\{\tau\}: \ \ \s\in\cm_f(B)\}. $$
Indeed, such measure is exactly the Clark measure $\s=\s_\a$ with $\a=B^{-1}(\tau)$, $|\supp\s_{\max}|=n$, and
$$ \frac1{\s_{\max}\{\tau\}}=\sum_{k=1}^d r_k\,\frac{1-|z_k|^2}{|\tau-z_k|^2}. $$

\begin{remark}
As a matter of fact, the above Clark measure solves the same extremal problem within the whole class $\cm(B)$.
Relation \eqref{weights} holds in the form
$$ \frac1{s_q}=\sum_{k=1}^d r_k\,\frac{1-|z_k|^2}{|\tau-z_k|^2}+|\o'(\tau)|, $$
where $\o'$ is the angular derivative of $\o$ (cf. \cite[Section 9.2]{CiMaR}).
\end{remark}

Here is another simple property of measures $\s\in\cm_f(B)$.

\begin{proposition}\label{pr3}
Let $\{t_j\}_{j=1}^p$ be an arbitrary set of distinct points on $\bt$. There is a measure $\s\in\cm_f(B)$ such that
\begin{enumerate}
  \item $\{t_j\}\in\supp\s$;
  \item $|\supp\s|\le n+p-1$.
\end{enumerate}
\end{proposition}
\begin{proof}
The proof is based on the interpolation with FBP's (see, e.g., \cite[Theorem~1]{jorus}): there is a FBP $\o$ so that
$\deg\o\le p-1$ and
$$ \o(t_j)=B^{-1}(t_j), \qquad j=1,\ldots, p. $$
The corresponding measure $\s$ in \eqref{alfin} is the one we need.
\end{proof}

It turns out that the intersection of supports of two different measures from $\cm_f(B)$ can not be too large.
Put
$$ \cm_{n+k}(B):=\{\s\in\cm_f(B): \ \ |\supp\s|=n+k\}, \quad k=0,1,\ldots. $$

\begin{lemma}\label{le1}
Let $\s_j\in\cm_{n+p_j}(B)$, $j=1,2$, and let
$$ |\supp\s_1\cap\supp\s_2|\ge p_1+p_2+1. $$
Then $\s_1=\s_2$.
\end{lemma}
\begin{proof}
Let $\o_j$ be the FBP's for $\s_j$ in \eqref{alfin}, $\deg\o_j=p_j$, $j=1,2$. Let \newline
$\z_1,\ldots,\z_{p_1+p_2+1}\in \supp\s_1\cap\supp\s_2$,
so, by \eqref{support},
$$ \o_1(\z_l)=\o_2(\z_l), \qquad l=1,2,\ldots, p_1+p_2+1. $$
Note that
$$ \o_j(z)=\g_j\,\frac{Q_j(z)}{Q_j^*(z)}\,, \quad j=1,2, $$
where $\g_j$ are unimodular constants, $Q_j$ are algebraic polynomials, $Q_j^*$ are the reversed polynomials, and
$$ \deg Q_j=p_j, \qquad \deg Q_j^*\le p_j, \qquad j=1,2. $$
We see that for the polynomial
$$ Q(z)=\g_1Q_1(z)Q_2^*(z)-\g_2Q_2(z)Q_1^*(z), \qquad \deg Q\le p_1+p_2, $$
the relations
$$ Q(\z_l)=0, \qquad l=1,2,\ldots, p_1+p_2+1 $$
hold, so $Q=0$, $\o_1=\o_2$, and $\s_1=\s_2$ (see Remark \ref{rem1}).
\end{proof}

\begin{corollary}
If $\s_j\in\cm_n(B)$, $j=1,2$, and $\supp\s_1\cap\supp\s_2\not=\emptyset$, then $\s_1=\s_2$.
If $\s_j\in\cm_{n+k}(B)$, $k=0,1,\ldots,n-1$, and $\supp\s_1=\supp\s_2$, then $\s_1=\s_2$.
\end{corollary}

\section{Extreme points of $\cm(\thth)$ for finite dimensional model spaces}
\label{s2}

We begin with the result which provides an upper bound in  \eqref{mainres}. It can be viewed as
a counterpart of \cite[Theorem 2.3.4]{Akh65} for the classical moment problem.

\begin{proposition}\label{pr2}
Let $\s\in\cm_{ext}(B)$. Then $\s\in\cm_f(B)$ and $|\supp\s|\le 2n-1$.
\end{proposition}
\begin{proof}
Assume first that $z_d=0$. Define a system of real valued, linearly independent functions on $\bt$, accompanying \eqref{basis1}
\begin{equation*}
\begin{split}
x_{k,j}(t) &:=\re \p_k^j(t), \quad y_{k,j}(t):=\im \p_k^j(t), \quad j=1,\ldots, r_k, \quad  k=1,\ldots, d-1,  \\
x_{d,j}(t) &:=\re t^j, \quad \quad \ \ y_{d,j}(t):=\im t^j, \quad \quad \ \ j=1,\ldots, r_d-1, \quad x_{d,0}=1.
\end{split}
\end{equation*}
We arrange  them in a sequence $\{v_l\}_{l=1}^{2n-1}$, and denote by $E$ their complex, linear span
$$ E:=\lc_{1\le l\le 2n-1}\{v_l\}, \qquad \dim E=2n-1. $$
Clearly, $t^l\in E$ for $|l|\le r_d-1$, and
$$ \p_k^j=x_{k,j}+iy_{k,j}\in E, \qquad \overline{\p_k^j}=x_{k,j}-iy_{k,j}\in E $$
(or $e_m, \overline{e_m}\in E$) for the appropriate values of $k,j,m$.
It is a matter of a direct computation to make sure that the product $e_m\overline{e_l}\in E$,
$m,l=1,\ldots,n$. For instance,
\begin{equation*}
\begin{split}
\p_p(t)\overline{\p_q(t)} &=\frac1{(1-\bar z_p t)(1-z_q\bar{t})}=\frac{\p_p(t)+\overline{\p_q(t)}-1}{1-\bar z_p z_q}\,, \\
\p_p^2(t)\overline{\p_q(t)} &=\frac1{(1-\bar z_p t)^2(1-z_q\bar{t})}=\frac{\p_p^2(t)+\p_p(t)\overline{\p_q(t)}-\p_p(t)}{1-\bar z_p z_q}\,,
\end{split}
\end{equation*}
etc. The rest is a simple induction. We conclude, thereby, that $f\overline{g}\in E$ for each $f,g\in K_B$.

Assume next, that $|\supp\s|\ge 2n$. Then the inclusion $E\subset L^1_\s(\bt)$ is {\it proper}, so there is
a nontrivial, linear functional $\Phi_0$ on $L^1_\s(\bt)$, $\|\Phi_0\|\le1$, vanishing on $E$. Equivalently,
there is a function $\p_0\in L^\infty_\s(\bt)$ such that $|\p_0|\le1$ $[\s]$-almost everywhere, and
$$ \int_\bt x_{k,j}(t)\p_0(t)\,\s(dt)=\int_\bt y_{k,j}(t)\p_0(t)\,\s(dt)=0 $$
for all appropriate values of $j,k$. Since the functions $x_{j,k}$, $y_{j,k}$ are real valued, the function $\p_0$
can be taken real valued as well.

Consider now two measures $\s_{\pm}(dt):=(1\pm\p_0)\,\s(dt)$, $\s_{\pm}\in\cm_+(\bt)$. By the construction,
$\s_{\pm}\in\cm(B)$, and the representation $2\s=\s_+ + \s_-$ is nontrivial. Hence, $\s$ is not an extreme point
of $\cm(B)$, as claimed.

It remains to examine the general case when $B(0)\not=0$. The standard trick with the change of variables
(see, e.g., \cite[pp. 140--141]{Si1}) reduces this case to the one discussed above. Given $a\in\bd$, put
$$ b_a(z):=\frac{z+a}{1+\bar a z}\,, \qquad B_a(z):=B(b_a(z)), \qquad \o_a(z):=\o(b_a(z)). $$
If we replace $z$ with $b_a(z)$ in \eqref{thal}, we have
$$ \frac{1+B_a(z)\o_a(z)}{1-B_a(z)\o_a(z)}=i\b+\int_{\bt}\,\frac{t+b_a(z)}{t-b_a(z)}\,\s(dt), $$
and since
$$ \frac{t+b_a(z)}{t-b_a(z)}=i\b_{a,t}+\frac{1-|a|^2}{|t-a|^2}\,\frac{b_a(t)+z}{b_a(t)-z}\,, $$
we come to
$$ \frac{1+B_a(z)\o_a(z)}{1-B_a(z)\o_a(z)}=
i\b_a+\int_\bt \frac{1-|a|^2}{|t-a|^2}\,\frac{b_a(t)+z}{b_a(t)-z}\,\s(dt)=
i\b_a+\int_\bt \frac{\tau+z}{\tau-z}\,\s^{(a)}(d\tau). $$
It is clear that the map $\s\to\s^{(a)}$ is a bijection of $\cm(B)$ onto $\cm(B_a)$, which is also the bijection
between $\cm_{ext}(B)$ and $\cm_{ext}(B_a)$. Obviously, it is a bijection between $\cm_f(B)$ and $\cm_f(B_a)$,
and in this case $|\supp\s|=|\supp\s^{(a)}|$. But $B_a(0)=0$ with $a=z_d$, so the above argument applies.
The proof is complete.
\end{proof}

\smallskip

\noindent
{\it Proof of Theorem \ref{mainth}}.

It remains to show that each measure $\s\in\cm_{n+k}(B)$, $k=0,1,\ldots,n-1$ is the extreme point of $\cm(B)$.
Indeed, let $2\s=\s_1+\s_2$, then
$$ \s_j\in\cm_{n+p_j}, \qquad j=1,2, \qquad 0\le p_1,p_2\le k. $$
Since $\supp\s=\supp\s_1\cup\supp\s_2$, we have
$$ |\supp\s|=|\supp\s_1|+|\supp\s_2|-|\supp\s_1\cap\supp\s_2|, $$
or
$$ |\supp\s_1\cap\supp\s_2|=n+p_1+n+p_2-n-k=n+p_1+p_2-k\ge p_1+p_2+1. $$
By Lemma \ref{le1}, $\s_1=\s_2$, so $\s$ is the extreme point of $\cm(B)$, as claimed.

\section{Extreme points of $\cm(\thth)$ for generic inner functions}
\label{s3}

The case of generic inner functions is much more delicate.

Let us define a binary operation in the Schur class, which corresponds to taking a half-sum of measures under transformation $\cf$.
Precisely, given two Schur functions $s_1$, $s_2$, we denote by $s:=(s_1\circ s_2)_\thth$ the operation so that
\begin{equation}\label{ieppro}
\cf\Bigl((s_1\circ s_2)_\thth\Bigr)=\frac{\cf(s_1)+\cf(s_2)}2\,.
\end{equation}
It is a matter of elementary computation based on the relation
$$ \frac{1+\thth(z)s(z)}{1-\thth(z)s(z)}=\frac12\lp\frac{1+\thth(z)s_1(z)}{1-\thth(z)s_1(z)}+\frac{1+\thth(z)s_2(z)}{1-\thth(z)s_2(z)}\rp\,,
\quad s=(s_1\circ s_2)_\thth, $$
to check that
\begin{equation}\label{dthprod}
(s_1\circ s_2)_\thth:=\frac{s_0-\thth s_1 s_2}{1-\thth s_0}\,, \qquad s_0:=\frac{s_1+s_2}2\,.
\end{equation}
$(s_1\circ s_2)_\thth$ will be called a {\it $\thth$-product} of $s_1$ and $s_2$.

We list the main properties of the $\thth$-product in the statement below.

\begin{proposition}\label{thprod}
Let $\thth$ be a nonconstant inner function.

\noindent
{\rm (i)}. \  $\circ$ is a binary operation on the Schur class, which is idempotent, that is, $s\circ s=s$.

\noindent
{\rm (ii)}. \  $s=(s_1\circ s_2)_\thth$ is an inner function if and only if so are both $s_1$ and $s_2$.

\noindent
{\rm (iii)}. \  The equality holds
  \begin{equation}\label{sara}
  s=s_0+\thth h, \qquad s_0=\frac{s_1+s_2}2, \quad h\in H^\infty.
  \end{equation}
\end{proposition}
\begin{proof}
(i). Since $1-\thth s_0$ is an outer function \cite[Corollary II.4.8]{Gar}, $(s_1\circ s_2)_\thth$ belongs to the Smirnov class, so one has to verify that
$$ |(s_1\circ s_2)_\thth(t)|\le1 $$
for a.e. $t\in\bt$. Indeed,
\begin{equation*}
\begin{split}
|1-\thth s_0|^2 &=1+|s_0|^2-\re(\thth s_1+\thth s_2), \\
|s_0-\thth s_1s_2|^2 &=|s_0|^2+|s_1s_2|^2-|s_1|^2\re(\thth s_2)-|s_2|^2\re(\thth s_1),
\end{split}
\end{equation*}
so
\begin{equation*}
\begin{split}
&{} |1-\thth s_0|^2-|s_0-\thth s_1s_2|^2=1-|s_1s_2|^2-\re(\thth s_1)(1-|s_2|^2)-\re(\thth s_2)(1-|s_1|^2) \\
&=1-|s_1s_2|^2-|s_1|(1-|s_2|^2)-|s_2|(1-|s_1|^2) \\
&+ (|s_1|-\re(\thth s_1))(1-|s_2|^2)+(|s_2|-\re(\thth s_2))(1-|s_1|^2) \\
&=(1-|s_1s_2|)(1-|s_1|)(1-|s_2|)+(|s_1|-\re(\thth s_1))(1-|s_2|^2) \\
&+(|s_2|-\re(\thth s_2))(1-|s_1|^2)\ge0,
\end{split}
\end{equation*}
as needed.

By definition \eqref{dthprod}, $(s\circ s)_\thth=s$ for each $s\in\css$, so the operation is idempotent.

(ii). If both $s_1$ and $s_2$ are inner functions then, by the above calculation, so is $s_1\circ s_2$.
Conversely, assume that $(s_1\circ s_2)_\thth$ is an inner function, but $|s_1|<1$ a.e. on a set $E\subset\bt$ of positive measure. It follows from the
above calculation that
$$ |s_2|=1, \qquad |s_2|-\re(\thth s_2)=0 $$
a.e. on $E$.  Hence $\thth s_2=1$ a.e. on the set of positive measure, so $\thth$ is a unimodular constant.

(iii). It follows directly from the definition that
\begin{equation*}
\begin{split}
s_0 &=s+\thth\,\frac{s_1s_2-s_0^2}{1-\thth s_0}=s+\thth h, \\
h &=\frac{s_1s_2-s_0^2}{1-\thth s_0}=-\frac14\,\frac{(s_1-s_2)^2}{1-\thth s_0}\,.
\end{split}
\end{equation*}
Since $1-\thth s_0$ is an outer function, $h$ belong to the Smirnov class, and moreover,
$$ |h(t)=|s_0(t)-s(t)|\le 2 $$
a.e. on $\bt$, so $h\in H^\infty$, as claimed.
\end{proof}

\begin{remark}
The $\thth$-product is a nontrivial operation already for $\thth=1$. It is clear from the definition, that for $s_2=\thth=1$
one has $(s_1\circ s_2)_\thth=1$ for any $s_1\in\css$.
\end{remark}

\begin{definition}
A function $s\in\css$ is called {\it $\thth$-prime} if
$$ s=(s_1\circ s_2)_\thth \ \ \Rightarrow \ \ s=s_1=s_2. $$
\end{definition}
It is clear from \eqref{ieppro} that $\s=\cf(s)\in\cm_{ext}(\thth)$ if and only if $s$ is $\thth$-prime.
Equivalently, $s\in\css_{ext}(\thth)$ if and only if $s$ is $\thth$-prime.

\begin{example}
The Lebesgue measure $m$ is the only one which knowingly belongs to $\cm(\thth)$ for any inner function $\thth$.
We show that it is never in $\cm_{ext}(\thth)$. Indeed, $m=\cf(0)$, so
$$ 0=s_1\circ s_2=\frac{s_0-\thth s_1s_2}{1-\thth s_0}, $$
which is equivalent to
$$ s_0=\thth s_1s_2, \quad s_1+s_2=2\thth s_1s_2, \quad s_2=\frac{s_1}{2\thth s_1-1}\,. $$
Take any $s_1$ with $\|s_1\|\le 1/3$, so
$$ |s_2(z)|=\frac{|s_1(z)|}{|2\thth s_1-1|}\le\frac1{3(1-2|s_1(z)|)}\le1, $$
We see that $s_2\in\css$, and there is a nontrivial factorization $0=s_1\circ s_2$, so $0$ is not $\thth$-prime.
\end{example}

\begin{example}
It is easy to see that the Clark measures $\s_\a=\cf(\a)\in\cm_{ext}(\thth)$ for each $\a\in\bt$ and any inner function $\thth$. Indeed,
assume that
$$ \a=\frac{\o_0-\thth \o_1\o_2}{1-\thth\o_0}\,, \quad \o_0=\frac{\a+\thth\o_1\o_2}{1+\a\thth}\,, $$
and so
$$ (\a-\o_0)(1+\a\thth)=\thth(\a^2-\o_1\o_2). $$
But both functions $\a-\o_0=\a(1-\bar\a\o_0)$ and $1+\a\thth$ are outer, and so is their product, whereas the left side has a nontrivial
inner factor. Hence
$$ \a^2=\o_1\o_2, \quad \o_0=\a, $$
which implies $\o_1=\o_2=\o=\a$ is $\thth$-prime, as claimed.
\end{example}

\smallskip

To determine whether a function $\o\in\css$ is $\thth$-prime, we apply uniqueness conditions in the celebrated Nehari problem,
with the notion of minifunction playing a key role.

Given a function $g\in \cb$, the unit ball of $L^\infty(\bt)$, the Nehari problem concerns a set
$$ N(g):=\{g+H^\infty\}\cap\cb, $$
which is nonempty, since $g\in N(g)$. Following Adamjan--Arov--Krein \cite{aak2}, we call $g$ an {\it undeformable minifunction}
if $N(g)=\{g\}$. Clearly, $\|g\|_\infty=1$ for each such $g$.

\begin{lemma}\label{minif}
Given a nonconstant inner function $\thth$ and $s\in\css$, let $s\bar\thth$ be an undeformable minifunction. Then $s$ is $\thth$-prime.
\end{lemma}
\begin{proof}
Let $s=s_1\circ s_2$, by Proposition \ref{thprod}, (iii),
$$ s\bar\thth=s_0\bar\thth+h, \qquad h\in H^\infty. $$
But $s\bar\thth$ is an undeformable minifunction, so $s=s_0$ a.e. Next,
$$ s_0=\frac{s_0-\thth s_1s_2}{1-\thth s_0}\,, \qquad s_0^2=\left(\frac{s_1+s_2}2\right)^2=s_1s_2, $$
and so $s_1=s_2$ a.e., and $s$ is $\thth$-prime, as needed.
\end{proof}

\smallskip

\noindent
{\it Proof of Theorem \ref{finblext}}.

In view of Lemma \ref{minif}, we need to show that $b\bar\thth$ is the undeformable minifunction, where $b$ is a FBP, and
the inner function $\thth$ is not. According to the result in \cite[Remark 3.2]{aak2}, a unimodular function $g$, i.e., $|g|=1$ a.e. on $\bt$,
is not an undeformable minifunction if and only if it admits the following representation
\begin{equation}\label{minifcr}
g(t)=\frac{\xi_+(t)}{t\xi_-(t)}\,, \qquad \xi_\pm\in H^2_\pm,
\end{equation}
a.e. on $\bt$.

Assume that \eqref{minifcr} holds for $g=b\bar\thth$ as above. Then $tb\xi_-=\thth\xi_+$, and so $\thth\xi_+\in K_{tb}$. But the latter model space
consists of rational functions, see \eqref{finmodsp}, and the function $\thth\xi_+$ is certainly not such. The contradiction completes the proof.
\hfill $\Box$

\begin{remark}
For particular inner functions $\thth$ the result can be obtained by simpler means. For instance, let $\thth$ have infinitely many zeros $\{z_j\}_{j\ge1}$.
Let a FBP $b=b_1\circ b_2$. Then, by \eqref{sara}, we have $b=b_0+\thth h$ and so
$$ b(z_j)=b_0(z_j), \qquad j=1,2,\ldots. $$
The uniqueness condition in the Nevanlinna--Pick interpolation problem ($b$ is the FBP), see \cite[Theorem I.2.2 and Corollary I.2.3]{Gar}, implies
$$ b=b_0=\frac{b_1+b_2}2\,, $$
which leads to $b_1=b_2=b$, as needed.

The same argument applies in the case when the underlying singular measure for $\thth$ contains an atom, with the boundary Nevanlinna--Pick problem instead,
see \cite[Theorem 1.3]{bol08}.
\end{remark}

\smallskip

There is a striking similarity between the Aleksandrov formula \eqref{thal} and the Nevanlinna parametrization for an indeterminate Hamburger moment problem,
see \cite[Theorem 3.2.2] {Akh65}. Recall that the Hamburger moment problem deals with measures $\mu$ on the real line having prescribed moments of all orders.
It is called indeterminate if the set of such measures is infinite. It turns out that all measures $\mu$ arise from the Nevanlinna formula
\begin{equation}\label{nevan}
\frac{A(z)u(z)+B(z)}{C(z)u(z)+D(z)}=\int_\br \frac{\mu(dx)}{x-z}\,, \qquad z\in\bc_+,
\end{equation}
when the independent parameter $u$ runs over the class of analytic on the upper half-plane $\bc_+$ functions with the nonnegative imaginary part ($u=\infty$ is
included). A measure $\mu$ is called a {\it canonical solution of order} $n=0,1,\ldots$ if the parameter $u$ in \eqref{nevan} is a real rational function of
order $n$.

The set of all solutions is clearly convex, so the extreme points show up. The result in \cite[Corollary 3.4.3]{Akh65} states that each canonical solution
of order $n$ is the extreme measure for the indeterminate Hamburger moment problem. Theorem \ref{finblext} is a direct counterpart of the latter result.

\smallskip

\noindent
{\it Proof of Theorem \ref{divis}}.

To show that each divisor $\p$ of $\thth$, $\p\not=\thth$, is $\thth$-prime, we again apply Lemma \ref{minif} with $g=\p\bar\thth=\bar\psi$,
$\psi$ is a nonconstant inner function.  Indeed, assume, on the contrary, that $\bar\psi$ is not an undeformable minifunction.
Then there is a nonzero function $f\in H^\infty$ such that
$$ \|\bar\psi-f\|_\infty=\|1-\psi f\|_\infty\le 1. $$
Hence, $\re \psi f\ge0$ on the disk $\bd$, and, by \cite[Corollary II.4.8]{Gar}, the function $\psi f$ is outer. But it has a nontrivial
inner factor $\psi$, so the contradiction justifies the claim.

For any nonzero $s\in\css$ the function $\thth s^2=(s\circ (-s))_\thth\notin\css_{ext}(\thth)$, so in particular, $\thth\notin\css_{ext}(\thth)$.
The proof is complete.                                               \hfill $\Box$

\smallskip

Finally, we show that there are non-inner Schur functions $w\in\css_{ext}(\thth)$, as long as $\thth$ is not a FBP.

\begin{proposition}
Assume that $\thth$ is not a FBP. Then there are outer Schur functions $w\in\css_{ext}(\thth)$.
\end{proposition}
\begin{proof}
There is a point $\z\in\bt$ so that $\thth$ does not admit an analytic continuation across any open arc $\gg$ centered at $\z$. Fix such arc,
and take an outer Schur function $w$ such that $w^{-1}\in H^\infty$ and $|w|=1$ a.e. on $\gg$.
Following the line of reasoning from \cite[Example IV.4.2]{Gar}, we show first that $w\bar\thth$ is an undeformable minifunction.
Indeed, assume that there is a nonzero function $g\in H^\infty$ so that
$$ \|w\bar\thth-g\|_\infty=\|w-\thth g\|_\infty\le1, $$
and hence
$$ |w-\thth g|=|w| |1-\thth g/w|=|1-\thth g/w|\le 1 $$
a.e. on $\gg$. So $\re \thth g/w\ge0$ a.e. on $\gg$, and, by \cite[Exercise II.14 (a)]{Gar}, the inner factor of the function $\thth g/w\in H^\infty$
admits an analytic continuation across $\gg$, that contradicts our assumption.

By Lemma \ref{minif}, $w$ is $\thth$-prime, as claimed.
\end{proof}

I thank A. Kheifets for the valuable remarks about minifunctions and the result of Adamjan--Arov--Krein.

\end{document}